\documentclass[a4paper,11pt,centertags,reqno,psamsfonts]{amsart}

\usepackage{letltxmacro}        
\usepackage{ifthen}		        
\usepackage[headings]{fullpage} 
\usepackage{setspace}		    
\usepackage[inline]{enumitem}   
\usepackage{array}				
\usepackage{fancybox}		    
\usepackage{graphicx}           
\usepackage{subfigure}		    
\usepackage{amsmath}	        
\usepackage{amsthm}				
\usepackage{amssymb}				
\usepackage{color}			
\usepackage{mathtools}			
\usepackage{mathrsfs}			
\usepackage{stmaryrd}			
\usepackage{yhmath}				
\usepackage{accents}			
\usepackage{extarrows}          
\usepackage{xfrac}              
\usepackage[all]{xy}			
\usepackage[round,comma,authoryear,sort&compress]{natbib} 

\usepackage{bm}

\newcommand{\N}{\mathbb{N}}
\newcommand{\Z}{\mathbb{Z}}
\newcommand{\Q}{\mathbb{Q}}
\newcommand{\R}{\mathbb{R}}

\newcommand{\C}[2]{
\ifthenelse{#1=0 \and #2=0}{\textsf{\upshape C}}
{\ifthenelse{#2=0}{\textsf{\upshape C}^{#1}}
{\textsf{\upshape C}^{#1,#2}}}
}

\renewcommand{\d}{\mathrm{d}}

\newcommand{\E}{\textsf{\upshape E}}

\renewcommand{\P}{\textsf{\upshape P}}

\newcommand{\filtration}[1]{\mathfrak{#1}}
\newcommand{\sigalgebra}[1]{\mathscr{#1}}

\makeatletter
\let\oldr@@t\r@@t
\def\r@@t#1#2{%
\setbox0=\hbox{$\oldr@@t#1{#2\,}$}\dimen0=\ht0
\advance\dimen0-0.2\ht0
\setbox2=\hbox{\vrule height\ht0 depth -\dimen0}%
{\box0\lower0.4pt\box2}}
\LetLtxMacro{\oldsqrt}{\sqrt}
\renewcommand*{\sqrt}[2][\ ]{\oldsqrt[#1]{#2}}
\makeatother

\theoremstyle{plain}
\newtheorem{theorem}{Theorem}
\newtheorem{lemma}[theorem]{Lemma}
\newtheorem{proposition}[theorem]{Proposition}

\theoremstyle{definition}

\theoremstyle{remark}
\newtheorem{remark}[theorem]{Remark}

\numberwithin{figure}{section}
\numberwithin{table}{section}

\graphicspath{{Figures/}}

\bibpunct{(}{)}{,}{a}{}{,}

\allowdisplaybreaks[4]

\begin{document}
\title{A one-dimensional diffusion hits points fast}
\author{Cameron Bruggeman}
\author{ Johannes Ruf}

\address{Cameron Bruggeman\\
Department of Mathematics\\
Columbia University}

\email{bruggeman@math.columbia.edu}

\address{Johannes Ruf\\
Department of Mathematics\\
University College London}

\email{j.ruf@ucl.ac.uk}



\date{\today}

\begin{abstract}
A one-dimensional, continuous, regular, and strong Markov process $X$ with state space $E$ hits any point $z \in E$ fast with positive probability. To wit, if ${\bm \tau}_z = \inf \{t \geq 0:X_{t} = z\}$, then $\P_\xi({\bm \tau}_z<\varepsilon)>0$ for all $\xi \in E$ and $\varepsilon>0$.
\end{abstract}

\maketitle

\section{Introduction}
Consider  a measurable function $\sigma: \R \mapsto \R \setminus \{0\}$  such that $1/\sigma^2$ is locally integrable.  Then  \citet{EngelbertSchmidt1981} guarantee the existence of a filtered probability space $(\Omega, \sigalgebra{F}, \filtration{F}, \P)$, 
 equipped with a Brownian motion $W = (W_t)_{t \geq 0}$, and the existence of a stochastic process $Z = (Z_t)_{t \geq 0}$ such that
\begin{align*}
	Z_t = \int_0^t \sigma(Z_s) \d W_s\,, \qquad t \geq 0
\end{align*}
holds.  Moreover, $Z$ is strong Markov and continuous. Let now $z \in \R$, $\varepsilon>0$, and ${\bm \tau}_z^Z$ denote the first hitting time of $z$ by $Z$.  Then we know that $\P({\bm \tau}_z^Z<\infty)>0$.     \citet{Mijatovic:personal} and  \citet{Karatzas_Ruf_2013} ask whether also $\P({\bm \tau}_z^Z<\varepsilon)>0$ holds for all $\varepsilon>0$.  Only a partial answer is provided: 
 If $1/\sigma^4$ is locally integrable (everywhere, apart from countably many points), then the answer is affirmative.
  
 This note answers the question affirmatively in a  general setup. To this end, we fix an open interval $E$ of $\R$, and denote its closure by  $\overline E$. We then consider a 
  one-dimensional  Markov process $X = (X_t)_{t \geq 0}$ with state space $E$ on the filtered space $(\Omega, \sigalgebra{F}, \filtration{F})$, along with a family of probability measures $(\P_\xi)_{\xi \in \overline E}$.  We  denote the death-time of $X$ by $\bm \zeta$.  We assume  that $X$ is strong Markov, regular, continuous on $[0, \bm \zeta)$, and $\lim_{t \nearrow \bm \zeta} X_t$ exists and satisfies $\lim_{t \nearrow \bm \zeta} X_t \notin E$ on $\{\bm \zeta < \infty\}$.  We set $X_{\bm \zeta +s} =  \lim_{t \nearrow \bm \zeta} X_t \in \overline E$ for all $s \geq 0$ on $\{\bm \zeta < \infty\}$.     If $Y = (Y_t)_{t \geq 0}$ is a stochastic process and ${\bm \rho}$ a stopping time, then $Y^{{\bm \rho}} = (Y^{{\bm \rho}}_t)_{t \geq 0} = (Y_{{\bm \rho} \wedge t})_{t \geq 0}.$  Furthermore, if $Y$ is a semimartingale, we let $[Y] = ([Y]_t)_{t \geq 0}$ denote the quadratic variation process of $Y$.
 
We now define the stopping times
\begin{align*}  
{\bm \tau}_z = \inf \{t \geq 0:X_{t} = z\}, \qquad z \in \overline E.
\end{align*} 
Since $X$ is regular, we have $\P_\xi(\bm \tau_z < \infty) > 0$ for all $\xi, z \in E$.
Throughout the note we shall fix a starting point $\xi \in E$ and a target point $z \in E$. 
We are now able to state the main result of this note.
 \begin{theorem}  \label{T:1}
 	For all $\varepsilon>0$, we have $\P_\xi({\bm \tau}_z < \varepsilon)>0$.
 \end{theorem}
 
 \begin{remark}
 We now provide some warnings concerning Thereom~\ref{T:1}.
 \begin{itemize}
 \item
 The continuity of $X$ is clearly important in Theorem~\ref{T:1}. For instance, the compensated Poisson process with state space $E = \R$ is strong Markov and regular, but the assertion of Theorem~\ref{T:1} does not hold for it.  
 \item If $X$ is  Brownian motion then Theorem~\ref{T:1} clearly holds.  If $X$ is only a local martingale,  the Dambis-Dumbins-Schwarz theorem yields the representation $X = B_{[X]}$ for some Brownian motion $B$ and Lemma~\ref{L:1} below yields that $[X]$ is strictly increasing.  However, $B$ and $[X]$ are usually not independent.  In particular, $[X]$ might slow down as $X$ approaches a point.  
  Thus, an argument for Theorem~\ref{T:1} that is based purely on a change of time is incomplete.
 \qed
 \end{itemize}
 \end{remark}

 After we had completed this note, Umut Cetin pointed out to us that Theorem~1 could also be derived from the arguments in Appendix II of \citet{Kotani:Watanabe:1982}.  We feel, however, that the arguments of this note are different and more direct (and cuter :-)).

\section{Proof of Theorem~\ref{T:1}}
Before proving Theorem~\ref{T:1}, we provide some auxiliary results.

\begin{lemma} \label{L1}
	Let $v: [0,\infty) \rightarrow [0, \infty)$ denote a nonnegative function with $v(0) = 0$ that satisfies $v(t+s) - v(t) \leq s$ for all $s,t \geq 0$. Then the first variation of $v|_{[0,t]}$ is bounded by $2t$, for each $t \geq 0$.
\end{lemma}
\begin{proof}
Note that $v$ can increase by at most $t$ on the interval $[0,t]$. This, in conjunction with the nonnnegativity of $v$, then yields that $v$ can drop by at most $t$ as well, and hence the bound of $2t$.  
\end{proof}

Recall that we have fixed a strong Markov process $X$ with state space $E$ and a starting point $\xi \in E$ for which the following results are formulated.

\begin{proposition}  \label{P:2}
	Let $v: \overline E\rightarrow \R$ be a measurable function and assume that the Markov process $X$ is also a continuous $\P_\xi$--local martingale.
	Then the function $[0,\infty) \ni t \mapsto v( X_t)$ is of finite first variation on compact subintervals of $[0,\infty)$,  $\P_\xi$--almost surely, if and only if $v$ is constant on $E$. 
\end{proposition}

 Since Proposition~\ref{P:2} is the core step of this note's argument we provide three different proofs.   
 
 \bigskip
 
 \emph{Preparation for the proofs of Proposition~\ref{P:2}.}
Clearly, $v$ being constant on $E$ implies that  $v(X_\cdot)$ is of finite first variation; thus it suffices to argue the reverse direction.  
Hence, from now on, we will assume that  $v(X_\cdot)$ is of finite first variation on compact subintervals of $[0,\infty)$.
Note that $v(X_\cdot)$ is of finite first variation variation on $\{\bm \zeta<\infty\}$.  If $\P_\xi(\bm \zeta = \infty) >0$  let  $(a_n)_{n \in \N}$ be a strictly decreasing sequence and $(b_n)_{n \in \N}$ a strictly increasing  sequence such that $E = \bigcup_{n \in \N} (a_n, b_n)$ and $\xi \in (a_1,b_1)$. 
Moreover, let 
\begin{align*}
	\bm \zeta_n = \inf \{t \geq 0: X_t \notin (a_n, b_n)\}, \qquad n \in \N.
\end{align*}
Then we have $\bm \zeta_n < \infty$ and $v(X_\cdot^{\bm \zeta_n})$ is of finite first variation for each $n \in \N$.  Moreover, note that $v$ is constant on $E$ if and only if $v$ is constant on $(a_n, b_n)$ for each $n \in \N$. Thus, we shall assume, without loss of generality, that $v(X_\cdot)$ is of finite first variation.

Next, observe that the Dambis-Dumbins-Schwarz theorem  yields the existence of a Brownian motion $B = (B_t)_{t \geq 0}$ with $B_0 = \xi$, possibly on an extension of the probability space, such that $X = B_{[X]}$; see, for instance, Theorem~V.1.7 in \citet{RY}.   
Then, with $\bm \rho = [X]_{\zeta}$, the process
$v(B_\cdot^{\bm \rho})$ is of finite first variation. 

 \bigskip

The first proof relies on an application of the  It\^o-Meyer-Tanaka formula.
\begin{proof}[Proof~I of Proposition~\ref{P:2}]
Proceeding as in Section~5 in \citet{Cinlar1980} we observe that $v$ is a so called semimartingale function for a Brownian motion killed when hitting the boundary of $E$ and thus, $v$ is locally the difference of two convex functions.  More precisely,  
with  $(a_n)_{n \in \N}$ and $(b_n)_{n \in \N}$ as above, $v|_{[a_n, b_n]}$ is the difference of two convex functions.  It then suffices to prove that $D^- v|_{[a_n, b_n]} = 0$, where  $D^- v|_{[a_n, b_n]}$ denotes its left derivative, for each $n \in \N$.   To this end, 
   let 
\begin{align*}
	\bm \rho_n = \inf \{t \geq 0: B_t \notin (a_n, b_n)\}, \qquad n \in \N.
\end{align*}
Then the It\^o-Meyer-Tanaka formula  yields  $$v(  B_{\cdot}^ {\bm \rho_n} ) = v(\xi) + \int_0^{\cdot \wedge {\bm \rho_n} } D^- v|_{[a_n, b_n]}(B_t) \mathrm d B_t + A_{\cdot}^{\bm \rho_n}, \qquad n \in \N,$$
 where $A = (A_t)_{t \geq 0}$ is a process of finite first variation.  Since  $v( B_{\cdot}^ {\bm \rho_n} )$ is of finite first variation we obtain   $ \int_0^{\cdot \wedge {\bm \rho_n} } (D^- v|_{[a_n, b_n]}(B_t))^2 \mathrm d t = 0$, and thus $D^- v|_{[a_n, b_n]} = 0$ for each $n \in \N$, as desired.
\end{proof}

We remark that \citet{Aboulaich:Stricker} provide a similar proof.
The next proof has been suggested by Vilmos Prokaj, to whom we are very grateful. The proof requires the additional assumption that $v$ is of finite first variation and uses local time of Brownian motion.
\begin{proof}[Proof~II of Proposition~\ref{P:2}]
 Let  $N(x,y)$ denote the number of upcrossings of $[x,y]$ made by $B^{\bm \rho} $ for all $x,y \in \R$ with $x<y$.   Moreover, let $L_{\bm \rho}(x)$ denote the local time of $B^{\bm \rho}$ at $x \in \overline E$, fix $\varepsilon > 0$, and pick some  sufficiently small $\delta>0$, possibly depending on $\omega \in \Omega$,  such that
\begin{align*}
|\delta N(x, x+\delta) - L_{\bm \rho} (x)| \leq \varepsilon
\end{align*}
for all $x \in \overline E$.  Such a $\delta$ exists almost surely, thanks to the uniform convergence of Theorem~2 in \citet{Chacon:1981}.   Next, define the sequence $({\bm \sigma}_k)_{k \in \N_0}$ of stopping times inductively by  ${\bm \sigma}_0=0$ and $${\bm \sigma} _{k+1} = {\bm \rho} \wedge \inf\{ t > {\bm \sigma}_k : |B_t - B_{{\bm \sigma}_k}|=\delta \}.$$

Suppose that the first variation $\Xi$ of the  function $v(B_{\cdot}^{\bm \rho} )$    is finite almost surely, directly implying that $v$ is continuous on $E$.
Then we have
\begin{align*}
	\delta\, \Xi &\geq \delta \sum_{k \in \N} |v(B_{{\bm \sigma}_{k+1} \wedge {\bm \rho} }) - v(B_{{\bm \sigma}_k \wedge {\bm \rho} })| \geq \sum_{i \in \Z, (i \delta, i\delta +\delta) \subset E} |v( i\delta +\delta) - v(i\delta)| \delta N(i\delta , i\delta +\delta)\\
		&\geq \sum_{i \in \Z,  (i \delta, i\delta +\delta) \subset E} |v( i\delta +\delta) - v(i \delta)|  (L_{\bm \rho}(i \delta) - \varepsilon).
\end{align*}
Letting now $\delta$ tend to zero and using the continuity of $L_{\bm \rho}$, argued in Theorem~VI.1.7 in \citet{RY}, note that
\begin{align*}
	0 = \lim_{\delta \downarrow 0} \delta\, \Xi  \geq   \int_E L_{\bm \rho} (x) |\d v(x)| - \varepsilon  
	\,\textrm{TV}(v),
\end{align*}
where $\textrm{TV}(v)$ denotes the variation of $v$, which is finite by assumption.
Next, letting $\varepsilon$ tend to zero, taking expectations, and using Tonelli yields
$$\int_E \E_\xi[L_{\bm \rho} (x)] |\d v(x)| =0.$$
Since each expectation is strictly positive, we obtain that the function $v$ is constant on $E$.
\end{proof}

The third  proof follows a pathwise argument and relies less on the one-dimensional character of $X$.  The proof requires the additional assumption that $v(\xi) = 0$, $v$ is nonnegative, and there exists a $\P_\xi$--nullset $N$ such that for all $s,t \geq 0$ and $\omega \in \Omega \setminus N$ we have the upper-Lipschitz condition
	\begin{align}  \label{eq:190815}
		v\left(X_{t+s}(\omega)\right) - v\left(X_t(\omega)\right) \leq s.
	\end{align}
\begin{proof}[Proof~III of Proposition~\ref{P:2}]
Again, clearly $v$ is continuous on $E$. 
	Fix now some $\omega \in \Omega$ such that the function $f: [0,\bm \zeta(\omega)) \rightarrow \R$, $t \mapsto v(X_t(\omega))$ is of finite first variation, \eqref{eq:190815} holds, and $X(\omega)$ has no point of monotonicity (see Theorem~2.9.13 in \citet{KS1}). Then $f$ is continuous and  Theorem~3.23(b) in \citet{Folland}  yields that $f$ has a derivative $f'$ almost everywhere.  Levy's decomposition theorem, Hahn's decomposition theorem,  and Proposition~3.30 in \citet{Folland} yield the existence of two nonnegative measures $\mu_-$ and $\mu_+$, both singular with respect to each other and to Lebesgue measure, such that 
	\begin{align*}
		\mathrm d f = f' \mathrm d t - \mathrm d \mu_- + \mathrm d \mu_+.
	\end{align*}
	
Suppose now that $f'(t)>0$ for some $t > 0$. Then we must have
$f(t+h)-f(t)>0$ for all sufficiently small $h \in \R$, but then $t$ is a point of monotonicity of $X(\omega)$. This contradicts the choice of $\omega$. Thus $f' \leq 0$ and we get in the same way that $f' = 0$.  Therefore, 	$\mathrm d f = - \mathrm d \mu_- + \mathrm d \mu_+$ 
Since, on intervals, we have
$\mathrm d f \leq \mathrm d t$ thanks to the upper-Lipschitz condition we get  $\mu_+ \leq m + \mu_-$, where $m$ denotes the Lebesgue measure.   Thanks to a monotone class argument we also get $\mu_+(D) \leq m(D) + \mu_-(D)$ for all $D \in \mathcal B$, the Borel sigma algebra of $[0,\infty)$.  Thus, $\mu_+$ is both absolutely continuous  and singular with respect to $m + \mu_-$, and we get $\mu_+ = 0$.
Finally, since $f\geq 0$ and $f(0) = 0$, we have $\mu_-=0$, and so $f$ is constant.
\end{proof}

\begin{lemma}  \label{L:1}
	 Assume that the Markov process $X$ is also a continuous $\P_\xi$--local martingale. Then the quadratic variation process $[X]$ is $\P_\xi$--almost surely strictly increasing on $[ 0 , \bm \zeta)$.
\end{lemma}
\begin{proof}
	Proposition~III.3.13  and the discussion proceeding it in \citet{RY} yield that $X$ cannot be constant on an interval. Proposition~IV.1.13 in \citet{RY} then yields the statement.
\end{proof}

Before stating the next lemma we introduce some notation. Assume that $E$ is of the form $E = (a,b)$ for $a,b \in \overline \R$ with $a<b$. For each $x \in \overline E$ we now define the deterministic function $u_x : \overline E\mapsto [0,1]$ by 
\begin{align}  \label{eq:u}
 	u_x(y)= 1 \wedge \inf \{ t \geq 0: \P_x(\bm \tau_y \leq t) >0 \}, \,\,\, y \in  E; \quad u_x(a) = \lim_{y \searrow a} u_x(y); \quad u_x(b) = \lim_{y \nearrow b} u_x(y).
\end{align}
Note that  $u_x$ is nonincreasing before $x$ and nondecreasing after $x$; thus, in particular, the limits in \eqref{eq:u} always exist, for each $x \in \overline E$. Moreover, $u_x$ is
nonnegative, 
 of finite first variation, and satisfies $u_x(x) = 0$, for each $x \in  E$.   Observe that an equivalent formulation of Theorem~\ref{T:1} is the statement that $u_\xi$ is constant.

\begin{lemma} \label{L2} Assume that the Markov process $X$ is also a continuous $\P_\xi$--local martingale. 
The  function $u_\xi$, given in \eqref{eq:u}, satisfies the following two claims. 
\begin{enumerate}[label={\rm(\roman{*})}, ref={\rm(\roman{*})}]
	\item\label{T:1:P:C1} $u_\xi$ is continuous;
	\item\label{T:1:P:C2}  there exists a $\P_\xi$--nullset $N$ such that for all $s,t \geq 0$ and $\omega \in \Omega \setminus N$ we have
	\begin{align*}
		u_\xi(X_{t+s}(\omega)) - u_\xi(X_t(\omega)) \leq s.
	\end{align*}
	\end{enumerate}
\end{lemma}
\begin{proof}
To start, for all $x, w \in \overline E$, we have the triangle inequality
\begin{align}   \label{eq:u:add}
	u_x(\cdot) \leq u_x(w) + u_w(\cdot).
\end{align}
Indeed,  this is clear if either one of the two summands equals one. To see the distributional property of \eqref{eq:u:add} otherwise, fix $x, w,y \in \overline E$ and assume for the moment that the underlying probability space is the canonical one; see Section~I.3 in \citet{RY}. 
Then, for each path $\omega$ we have
the weak inequality $\bm \tau_y(\omega) \leq \bm \tau_w(\omega)  + \bm \tau_y(\theta_{\bm \tau_w(\omega)}(\omega))$, where $\theta$ denotes the shift operator; that is $\theta_t(\omega)(\cdot) = \omega(t+\cdot)$ for all $t \geq 0$, see also the discussion on page~104 in \citet{RY}. Fix now $\varepsilon > 0$ and $t_1 = u_x(w) + \varepsilon/2$ and $t_2 = u_w(y)+ \varepsilon/2$. 
Then we have
\begin{align*}
	\P_x(\bm \tau_y \leq t_1 + t_2) &\geq \P_x(\bm \tau_w +  \bm \tau_y(\theta_{\bm \tau_w}) \leq t_1 + t_2)
		\geq  \P_x(\bm \tau_w \leq t_1;   \bm \tau_y(\theta_{\bm \tau_w}) \leq  t_2)\\
		&= \P_x(\bm \tau_w \leq t_1) \P_w( \bm \tau_y \leq  t_2) > 0,
\end{align*}
where the equality follows the strong Markov property of $X$ and the last inequality follows from the definition of $t_1$ and $t_2$.  This yields directly that $u_x(y) \leq t_1 + t_2 =  u_x(w) + u_w(y) + \varepsilon$. Letting $\varepsilon$ tend to zero then gives \eqref{eq:u:add}.

Claim~\ref{T:1:P:C1}:   
First, for any $w \in E$, the continuity of $u_w$ at $w$ follows from the fact that  $X$ is not constant on any interval (see the proof of Lemma~\ref{L:1}), in conjunction with the strong Markov property.   Let us now study the continuity of $u_\xi$ at some $y \in E$.  Without loss of generality, we may assume that $y>\xi$. The right-continuity then follows from \eqref{eq:u:add} and the continuity of $u_y$ at $y$.
For the left-continuity of $u_\xi$ at $y$, Section~3.3 in \citet{ItoMcKean} or Lemma~4.1, in particular (4.5), in \citet{Karatzas_Ruf_2013} also hold for the case of the regular, strong Markov process $X$, thanks to Lemma~\ref{L:1}. Thus,  for each $\varepsilon>0$ there exists $w \in  (\xi,y)$ such that $ \P_w({\bm \tau}_y \leq \varepsilon) >0$. 
The left-continuity of $u_\xi$ at $y$ then follows by another application of  \eqref{eq:u:add}.

Claim~\ref{T:1:P:C2}:  
Assume first that there exists some $t \geq 0$ such that $\P_{w}(u_w(X_{t}) >  t )>0$ for some $w \in  E$.   This then implies that there exists some $y \in E$ such that $u_w(y)>t$ and $\P_{w}(X_{t} >  y )>0$ if $y>\xi$ and $\P_{w}(X_{t} <  y )>0$ if $y<\xi$, respectively. This, in conjunction with the continuity of $X$, however, contradicts the definition of $u_w$ in \eqref{eq:u}. We therefore have
 \begin{align} \label{L2:eq1}
	 	\P_{w}\left(u_w(X_{t}) \leq  t \right)= 1 \qquad \text{for all  $t \geq 0$ and $w \in E$.}
\end{align}
Fix now $q_1, q_2 \in \Q$. Conditioning and the strong Markov property of $X$ then yield that 
 $\P_\xi(u_\xi(X_{q_1 + q_2}) - u_\xi(X_{q_1}) \leq q_2) = 1$  if  
 \begin{align}   \label{L2:eq2}
 	\left.\P_{w}\left(u_\xi(X_{q_2})  - u_\xi(w) \leq  q_2 \right)\right|_{w = X_{q_1}} =1 \text{ holds $\P_\xi$--almost surely}.
 \end{align}
We now note that \eqref{eq:u:add} and \eqref{L2:eq1} imply \eqref{L2:eq2}.
The claim then follows from the continuity of $u_\xi$ and $X$.  
\end{proof}

We are now ready to prove this note's main result.

\begin{proof}[Proof of Theorem~\ref{T:1}]
	Thanks to Propositions~VII.3.2, VII.3.4, and VII.3.5 in \citet{RY} we may assume,  without loss of generality,  that $X$ is in natural scale and thus a $\P_\xi$--local martingale.  
Next,  we recall the function $u_\xi$, given in \eqref{eq:u}. 
Now Lemma~\ref{L1}, in conjunction with Lemma~\ref{L2}\ref{T:1:P:C2},  yields that the function 
$[0,\infty) \ni t \mapsto u_\xi(X_t)$ 
has finite first variation on compact subintervals of $[0,\infty)$,  $\P_\xi$--almost surely.     Proposition~\ref{P:2}  now implies that $u_\xi$ is constant.  This yields that $u_\xi(z) = u_\xi(\xi) = 0$,  and thus, the assertion of the theorem follows.
\end{proof}

\bigskip

\section*{Acknowledgements}
We thank Stefan Ankirchner, Umut Cetin, Leif D\"oring, Ben Hambly,  Aleks Mijatovi\'c, Nicolas Perkowski, Vilmos Prokaj, Mykhaylo Shkolnikov, and Mikhail Urusov for many helpful discussions on the subject matter of this note.  We are grateful to Tomoyuki Ichiba and Ioannis Karatzas for a careful reading of an earlier version. J.R.~acknowledges generous support from the Oxford-Man Institute of Quantitative Finance, University of Oxford.

\bibliography{aa_bib}
\bibliographystyle{apalike}
\end{document}